\documentclass[11pt]{article}
\usepackage[margin=1in]{geometry}
\usepackage{amsmath,amssymb,amsthm,mathtools,booktabs,microtype}
\usepackage[hidelinks]{hyperref}
\usepackage{enumitem}
\setlist{nosep}
\newtheorem{theorem}{Theorem}
\newtheorem{lemma}[theorem]{Lemma}

\newtheorem{corollary}[theorem]{Corollary}
\theoremstyle{remark}

\newcommand{\Z}{\mathbb Z}
\newcommand{\1}{\mathbf 1}
\newcommand{\T}{\mathcal T}

\title{Latin Squares with Few Transversals}
\author{Zur Luria}
\date{}

\begin{document}
\maketitle
\vspace{-1.5em}

\begin{abstract}
Let $t(n)$ denote the minimum number of transversals in a Latin square of odd order $n$. Improving upon a recent bound of Dai, Divoux and Kelly, we prove that for every $n$ such that $n \equiv 3 \pmod 6$,
\[
t(n) \leq \left( \left(1+o(1)\right) \frac{2n}{3e^2}\right)^n .
\]
Our proof is based on a family of $3 \times 3$ block Latin squares whose transversals are constrained to either lie entirely in the diagonal blocks or avoid them altogether. 
\end{abstract}

\section{Introduction}

A transversal of a Latin square of order $n$ is a set of $n$ entries containing one entry from each row and column, and one occurrence of each symbol. When $n$ is even, the cyclic Latin square defined by
\[
L(i,j)=i+j\pmod n
\]
has no transversals. In striking contrast, the celebrated Ryser conjecture asserts that every Latin square of odd order has at least one transversal.

Following Wanless~\cite{WanlessSurvey}, let $T(n)$ and $t(n)$ denote, respectively, the maximum and minimum number of transversals in a Latin square of order $n$. In his survey on transversals in Latin squares, Wanless explicitly raised the problem of obtaining nontrivial upper bounds on $t(n)$ for odd $n$.

The corresponding extremal problem for the maximum is by now well understood. An upper bound of Taranenko, together with a matching probabilistic construction of Glebov and Luria, gives
\[
 T(n)=\left((1+o(1))\frac{n}{e^2}\right)^n
 \qquad\text{\cite{Taranenko,GlebovLuria}.}
\]
This naturally raised the question of whether $t(n)$ might have the same exponential order. There was some evidence in this direction. Kwan showed that a uniformly random Latin square typically has at least
\[
 \left((1-o(1))\frac{n}{e^2}\right)^n
\]
transversals~\cite{Kwan}; together with the upper bound above, this determines the number of transversals of a typical Latin square on the exponential scale. Eberhard, Manners and Mrazovi\'c obtained the same exponential order for the Cayley tables of groups of odd order~\cite{EMM}.

The first exponential separation between $t(n)$ and $T(n)$ was obtained recently by Dai, Divoux and Kelly~\cite{DDK}. Their starting point was a construction of Egan and Wanless~\cite{EganWanless}: for every odd $q\geq 3$, there is a Latin square of order $3q$ containing a $(q-1)\times q$ subrectangle none of whose entries lies in a transversal. Combining this construction with an entropy bound for perfect matchings, Dai, Divoux and Kelly proved that
\[
 t(n)\leq \left(\frac{n}{e^{2.117}}\right)^n
\]
for all sufficiently large odd $n$ divisible by three.

In this note we give a simple construction which yields a stronger bound. For every odd $q$, we construct a Latin square $L_q$ of order $n=3q$ such that
\[
 |\T(L_q)|
 \leq
 \left((1+o(1))\frac{2n}{3e^2}\right)^n.
\]
Consequently, as $n\to\infty$ through integers satisfying $n\equiv3\pmod6$,
\[
 t(n)\leq
 \left(
 \frac{n}{e^{\,2+\log(3/2)-o(1)}}
 \right)^n.
\]

The construction is a $3\times3$ array of cyclic Latin squares of order $q$, differing from the direct product of the cyclic squares of orders $3$ and $q$ in only one block, which is shifted by one. The purpose of this modification is not to make that block itself difficult to traverse. Rather, the shift leaves a residue in the sum of the symbols of any transversal, forcing a global dichotomy: every transversal either lies entirely in the three diagonal blocks or avoids all three of them. Once this dichotomy is established, the required estimate follows directly from a bound for perfect matchings in regular hypergraphs~\cite{LuriaQueens}.

\begin{theorem}\label{thm:main}
As $n\to\infty$ through integers satisfying $n\equiv3\pmod6$,
\[
 t(n)\leq
 \left((1+o(1))\frac{2n}{3e^2}\right)^n.
\]
\end{theorem}

The author used OpenAI’s Sol model in developing this work. In particular, the model proposed the central construction and assisted with the exposition. The author verified all arguments and takes full responsibility for the contents of the paper.

\section{The construction}

Fix an odd integer $q$ and put $n=3q$. Index the rows, columns and symbols by $\Z_3\times\Z_q$. Define
\begin{equation}\label{eq:construction}
 L_q\big((i,a),(j,b)\big)
 =\big(i+j,\,a+b+\1_{\{i=j=0\}}\big),
\end{equation}
where the first coordinate is taken modulo $3$ and the second modulo $q$.
It is immediate from the definition that $L_q$ is a Latin square.

Let $C_q$ denote the cyclic Latin square whose $(a,b)$ entry is $a+b$, and let $C_q^+=C_q+1$, where the addition is modulo $q$. In block notation,
\begin{equation}\label{eq:block-square}
L_q=
\begin{pmatrix}
 (0,C_q^+) & (1,C_q) & (2,C_q)\\
 (1,C_q)   & (2,C_q) & (0,C_q)\\
 (2,C_q)   & (0,C_q) & (1,C_q)
\end{pmatrix}.
\end{equation}
Thus every block is cyclic, and only the upper-left block is shifted.

For example, when $q=5$ the two inner blocks are
\[
C_5=
\begin{pmatrix}
0&1&2&3&4\\
1&2&3&4&0\\
2&3&4&0&1\\
3&4&0&1&2\\
4&0&1&2&3
\end{pmatrix},
\qquad
C_5^+=
\begin{pmatrix}
1&2&3&4&0\\
2&3&4&0&1\\
3&4&0&1&2\\
4&0&1&2&3\\
0&1&2&3&4
\end{pmatrix}.
\]
Substituting these into~\eqref{eq:block-square} gives a Latin square of order $15$.

\section{Block profiles}

Partition $L_q$ into the nine $q\times q$ blocks appearing in~\eqref{eq:block-square}, and denote the block in position $(i,j)$ by $B_{ij}$, where $i,j\in\Z_3$.

We call the three blocks $B_{i0},B_{i1},B_{i2}$ the $i$th \emph{block row}, and define block columns similarly. We will also use the \emph{block symbol classes}: for $s\in\Z_3$, the block symbol class $s$ consists of the three blocks $B_{ij}$ with $i+j=s$. Every entry in these three blocks has first symbol coordinate $s$.

Now let $T$ be a transversal of $L_q$, and let
\[
x_{ij}=|T\cap B_{ij}|.
\]
Write $X=(x_{ij})$ for this $3\times3$ matrix, which we call the \emph{block profile} of $T$.

Since $T$ contains one entry from each row of $L_q$, it contains exactly $q$ entries in each block row. Likewise, it contains exactly $q$ entries in each block column. Finally, for each $s\in\Z_3$, there are exactly $q$ symbols whose first coordinate is $s$, and each occurs once in $T$. Hence $T$ contains exactly $q$ entries in block symbol class $s$. Thus
\begin{equation}\label{eq:line-sums}
 \sum_jx_{ij}=q,\qquad
 \sum_ix_{ij}=q,\qquad
 \sum_{i+j=s}x_{ij}=q.
\end{equation}

\begin{lemma}[The block profile]\label{lem:profile}
There are nonnegative integers $A,B,C$ with $A+B+C=q$ such that
\begin{equation}\label{eq:profile}
X=
\begin{pmatrix}
 A&C&B\\
 B&A&C\\
 C&B&A
\end{pmatrix}.
\end{equation}
\end{lemma}

\begin{proof}
Set
\[
A=x_{00},\qquad C=x_{01},\qquad B=x_{02}.
\]
The first block-row equation gives $A+B+C=q$. The matrix on the right-hand side of~\eqref{eq:profile} satisfies all the equations in~\eqref{eq:line-sums}, so it remains only to show uniqueness.

Subtract this matrix from $X$. Since both matrices satisfy the row, column and block-symbol sum equations in~\eqref{eq:line-sums}, their difference has sum zero on every block row, block column and block-symbol class. As its first row is zero and its column sums are zero, it has the form
\[
\begin{pmatrix}
0&0&0\\
p&r&s\\
-p&-r&-s
\end{pmatrix}.
\]
Now consider the three block-symbol classes, that is, the sets of positions with $i+j$ fixed modulo $3$. Their sums are zero, and hence
\[
s-r=0,\qquad p-s=0,\qquad r-p=0.
\]
Thus $p=r=s$. The second row sum is also zero, so $3p=0$, and therefore $p=r=s=0$.
\end{proof}

We now use the shift in the block $B_{00}$. For an entry $e\in L_q$, write
$r_2(e)$, $c_2(e)$ and $s_2(e)$ for the second coordinates of its row, column and symbol. By~\eqref{eq:construction},
\[
 s_2(e)
 \equiv
 r_2(e)+c_2(e)
 +\1_{\{e\in B_{00}\}}
 \pmod q.
\]
Summing this identity over the entries of a transversal $T$ gives
\begin{equation}\label{eq:checksum}
 \sum_{e\in T}s_2(e)
 \equiv
 \sum_{e\in T}r_2(e)
 +\sum_{e\in T}c_2(e)
 +x_{00}
 \pmod q.
\end{equation}

The first three sums are determined independently of $T$. Indeed, among the rows of $L_q$, each element of $\Z_q$ occurs exactly three times as a second coordinate, and $T$ uses every row exactly once. Hence
\[
\sum_{e\in T}r_2(e)
=
3\sum_{a=0}^{q-1}a.
\]
The same argument applies to columns and symbols. Since $q$ is odd,
\[
3\sum_{a=0}^{q-1}a
=
\frac{3q(q-1)}2
\equiv0\pmod q.
\]
Equation~\eqref{eq:checksum} therefore reduces to
\[
x_{00}\equiv0\pmod q.
\]
By Lemma~\ref{lem:profile}, $x_{00}=A$. Since $0\leq A\leq q$, we conclude that
\[
A\in\{0,q\}.
\]

This is the key effect of the shifted block. The block-profile constraints first force the three diagonal blocks to carry the same number $A$ of transversal entries. The shift in just one of them then detects $A$ modulo $q$, forcing that common value to be either zero or as large as possible.

\begin{corollary}[The dichotomy]\label{cor:dichotomy}
Every transversal of $L_q$ either lies entirely in the three diagonal blocks or avoids all three diagonal blocks.
\end{corollary}

\begin{proof}
If $A=q$, then $A+B+C=q$ implies $B=C=0$, so the transversal is supported entirely on the three diagonal blocks. If $A=0$, then all three diagonal entries of the block profile vanish, so the transversal avoids the diagonal blocks.
\end{proof}

\section{Counting transversals}

We use the following bound for perfect matchings in regular hypergraphs~\cite[Theorem~3.1]{LuriaQueens}.

\begin{theorem}\label{thm:matching}
Let $d$ be fixed, and let $H$ be a $d$-uniform, $k$-regular hypergraph on $N$ vertices, where $k\to\infty$ as $N\to\infty$ and the maximum codegree of $H$ is $o(k)$. Then the number of perfect matchings in $H$ is at most
\[
\left((1+o(1))\frac{k}{e^{d-1}}\right)^{N/d}.
\]
\end{theorem}

We represent $L_q$ as a $3$-partite $3$-uniform hypergraph whose vertex classes are the rows, columns and symbols of $L_q$, with one edge for each entry of the square. Under this correspondence, the transversals of $L_q$ are precisely the perfect matchings of the hypergraph. Moreover, any two vertices have codegree at most one.

Let $\T_{\rm diag}$ denote the set of transversals lying entirely in the three diagonal blocks. Restricting the hypergraph to the entries in these blocks gives a $q$-regular $3$-uniform hypergraph on $3n$ vertices. By Theorem~\ref{thm:matching},
\[
|\T_{\rm diag}|
\leq
\left((1+o(1))\frac{q}{e^2}\right)^n
=
\left((1+o(1))\frac{n}{3e^2}\right)^n.
\]

Similarly, let $\T_{\rm off}$ denote the set of transversals lying entirely in the six off-diagonal blocks. Restricting to these blocks gives a $2q$-regular $3$-uniform hypergraph on the same $3n$ vertices, and hence
\[
|\T_{\rm off}|
\leq
\left((1+o(1))\frac{2q}{e^2}\right)^n
=
\left((1+o(1))\frac{2n}{3e^2}\right)^n.
\]

By Corollary~\ref{cor:dichotomy}, every transversal of $L_q$ belongs to exactly one of these two classes. Therefore
\[
\begin{aligned}
|\T(L_q)|
&=
|\T_{\rm diag}|+|\T_{\rm off}|\\
&\leq
\left((1+o(1))\frac{n}{3e^2}\right)^n
+
\left((1+o(1))\frac{2n}{3e^2}\right)^n\\
&=
\left((1+o(1))\frac{2n}{3e^2}\right)^n.
\end{aligned}
\]
This proves Theorem~\ref{thm:main}.

\end{document}